\documentclass[11pt]{amsart}

\usepackage{microtype}
\usepackage{bbm}
\usepackage{mathtools} 

\swapnumbers

\usepackage{amssymb,amsfonts}
\usepackage[mathscr]{eucal}
\usepackage[alphabetic]{amsrefs}

\usepackage[ps,matrix,arrow,curve,cmtip]{xy}

\title{Frobenius pairs and Atiyah Duality}
\author{Charles Rezk}
\date{ \today}
\address{Department of Mathematics \\
University of Illinois at Urbana-Champaign \\ 
Urbana, IL}
\email{rezk@math.uiuc.edu}
\thanks{The author was supported under NSF grant DMS--1006054.}

\numberwithin{equation}{section}

\makeatletter
  \let\c@subsection\c@equation
\makeatother

\theoremstyle{plain}   

\newtheorem{prop}[subsection]{Proposition}
\newtheorem{cor}[subsection]{Corollary}

\theoremstyle{remark}
\newtheorem{rem}[subsection]{Remark}

\theoremstyle{plain}

\DeclareMathOperator{\id}{id}

\newcommand{\op}{{\operatorname{op}}}

\newcommand{\Hom}{{\operatorname{Hom}}}

\newcommand{\ra}{\rightarrow}

\newcommand{\xra}{\xrightarrow}

\newcommand{\R}{\mathbb{R}}

\newcommand{\sm}{\wedge} 

\newcommand{\dfn}{\textbf}

\def\defeq{\overset{\mathrm{def}}=}

\begin{document}

\newcommand{\one}{\mathbbm{1}}
\newcommand{\cc}{\mathfrak{C}}
\newcommand{\uHom}{\underline{\mathrm{Hom}}}

\begin{abstract}
We define a notion of ``Frobenius pair'', which is a mild generalization of
the notion of ``Frobenius object'' in a monoidal category.  We then
show that Atiyah duality for smooth manifolds can be encapsulated in
the statement that a 
certain collection of structure obtained from a manifold forms a
``commutative Frobenius pair'' in the stable homotopy category of spectra. 
\end{abstract}

\maketitle


\section{Introduction}

A \emph{Frobenius algebra} over a field $k$ is an associative
$k$-algebra equipped 
with a $k$-linear map $\lambda\colon A\ra k$ such that the pairing
$A\otimes_k A\ra k$ defined by $x\otimes y\mapsto \lambda(xy)$ is
non-degenerate.  One class of Frobenius algebras is produced by
\emph{Poincar\'e duality}: if $M$ is a closed compact manifold which is
orientable, the cohomology ring $H^*(M;k)$ admits the structure of a
Frobenius algebra.  There is a generalization of this notion to an
arbitrary monoidal category, which is called a \emph{Frobenius object}
\cite{kock-frobenius-algebras}.

If one wants to refine classical Poincar\'e duality to an arbitrary
generalized cohomology theory, one is lead to \emph{Atiyah duality}
\cite{atiyah-thom-complexes}.  
This states that for any closed compact manifold (not necessarily
orientable), there exists a non-degenerate pairing
\[
M^{-\tau}\sm M_+ \ra S^0
\]
in the stable homotopy category; here $M_+$ denotes the suspension
spectrum of $M$ with a disjoint basepoint, $M^{-\tau}$ denotes the Thom
spectrum associated to the stable normal bundle of $M$, and
``non-degenerate'' means that the adjoint map $M_+\ra
\uHom(M^{-\tau},S^0)$ is a weak equivalence, i.e., that $M_+$ and
$M^{-\tau}$ are Spanier-Whitehead dual.

We cannot in general say that $M_+$ is a Frobenius object in the
stable homotopy category; at best this can be done only if $M_+$
admits a stable framing.  By smashing with an appropriate
multiplicative cohomology theory, one can obtain a Frobenius object in
some category of module spectra; see
\cite{strickland-kn-local-duality-finite-groups} for a 
treatment along these lines. 

The goal of this note is to describe a mild generalization of the
notion of a Frobenius object in a monoidal category, which we call a
\emph{Frobenius pair}, and to show that Atiyah duality is naturally
encapsulated by this definition.  Notably, we use this formulation to
give a proof of Atiyah duality for smooth manifolds which has a rather
formal character, in the 
sense that we use only the existence of
geometric constructions such as the Pontryagin-Thom collapse map, the
equivalence of embeddings into Euclidean space of large dimension, and
standard formalities about Thom spectra.

I am sure that the notion which I have called  ``Frobenius pair'' has been
encountered elsewhere, 
though I've been unable to trace it in the literature.  I am also sure
that the proof of Atiyah duality sketched here is well known; some
constructions like the ones described here appear in
\cite{cohen-multiplicative-thom}.  The  proof given here actually grew out of
an attempt to generalize the proof of Poincar\'e duality in ordinary cohomology
described in Chapter 10 of
\cite{milnor-stasheff-characteristic-classes} to one valid for any
cohomology theory.
The notion of Frobenius pair and the proof of Atiyah
duality taken together make for a nice story, hence this note.

\section{Frobenius pairs}

The following definition takes place in a monoidal category $\cc$,
with unit object $\one$ and associative monoidal product
$\otimes\colon \cc\times\cc\ra \cc$.  The name ``Frobenius pair'' is meant
to be suggestive of ``Frobenius object'', except that instead of a
single object $A$ with structure, it consists of two objects $X$ and
$Y$, with some additional structure.

\subsection*{Definition of left Frobenius pair}

A \dfn{left Frobenius pair} in $\cc$ is data
$(X,Y,\eta,\mu,\psi,\epsilon,\delta,\phi)$ consisting of
\begin{itemize}
\item objects $X,Y$ of $\cc$,
\item morphisms
  \begin{align*}
\epsilon &\colon Y \ra \one,  
&\eta &\colon \one \ra X,
\\ 
 \delta &\colon Y\ra Y\otimes Y,   
&\mu &\colon X\otimes X\ra X,
\\
 \phi &\colon X\ra Y\otimes X,
&\psi &\colon X\otimes Y\ra Y,
  \end{align*}
\end{itemize}
such that the following hold.
\begin{enumerate}
\item [(1a)] The triple $(X,\eta,\mu)$ is an associative monoid object (with
  unit) in $\cc$.  That is, the diagrams
\[\xymatrix{
{X\otimes X\otimes X} \ar[r]^-{\mu\otimes X} \ar[d]_{X\otimes \mu}
& {X\otimes X} \ar[d]^{\mu}
& {\one\otimes X} \ar[r]^{\eta\otimes X} \ar@{=}[dr]
& {X\otimes X}  \ar[d]^{\mu}
& {X\otimes \one} \ar[l]_{X\otimes \eta} \ar@{=}[dl]
\\
{X\otimes X} \ar[r]_{\mu}
& {X}
&& {X}
}\]
commute.
\item [(1b)] The triple $(Y,\epsilon,\delta)$ is an associative comonoid object (with
  counit) in $\cc$.  That is, the diagrams
\[\xymatrix{
{Y\otimes Y\otimes Y}
& {Y\otimes Y} \ar[l]_-{\delta\otimes Y}
& {\one\otimes Y} \ar@{=}[dr]
& {Y\otimes Y} \ar[l]_{\epsilon\otimes Y} \ar[r]^{Y\otimes \epsilon}
& {Y\otimes \one} \ar@{=}[dl]
\\
{Y\otimes Y} \ar[u]^{Y\otimes \delta}
& {Y} \ar[u]_{\delta} \ar[l]^{\delta}
&& {Y} \ar[u]_{\delta}
}\]
commute.
\item [(2a)] The pair $(Y,\psi)$ defines a left $X$-module structure.  That is, the
  diagrams 
\[\xymatrix{
{X\otimes X\otimes Y} \ar[r]^-{\mu\otimes Y} \ar[d]_{X\otimes \psi}
& {X\otimes Y} \ar[d]^{\psi}
& {\one\otimes Y} \ar[r]^{\eta\otimes Y}  \ar@{=}[dr]
& {X\otimes Y} \ar[d]^{\psi}
\\
{X\otimes Y} \ar[r]_{\psi}
& {Y}
&& {Y}
}\]
commute.
\item [(2b)] The pair $(X,\phi)$ defines a left $Y$-comodule structure.  That
  is, the diagrams
\[\xymatrix{
{Y\otimes Y\otimes X}
& {Y\otimes X} \ar[l]_-{\delta\otimes X}
& {\one\otimes X} \ar@{=}[dr]
& {Y\otimes X} \ar[l]_{\epsilon\otimes X} 
\\
{Y\otimes X} \ar[u]^{Y\otimes \phi}
& {X} \ar[u]_{\phi} \ar[l]^-{\phi}
&& {X} \ar[u]_{\phi}
}\]
commute.
\item [(3a)] The map $\phi$ is a homomorphism of left $X$-modules.
  That is, the diagram
\[\xymatrix{
{X\otimes X} \ar[r]^-{X\otimes \phi} \ar[d]_{\mu}
& {X\otimes Y\otimes X} \ar[d]^{\psi\otimes X}
\\
{X} \ar[r]_-{\phi} 
& {Y\otimes X} 
}\]
commutes.
\item [(3b)] The map $\psi$ is a homomorphism of left $Y$-comodules.
  That is, the diagram
\[\xymatrix{
{Y\otimes Y}
& {Y\otimes X\otimes Y} \ar[l]_-{Y\otimes \psi}
\\
{Y} \ar[u]^{\delta}
& {X\otimes Y} \ar[l]^-{\psi} \ar[u]_{\phi\otimes Y}
}\]
commutes.
\end{enumerate}

As a consequence of these axioms, we have the following.
\begin{prop}
The map $\phi$ is a homomorphism of right $X$-modules, and the map
$\psi$ is a homomorphism of right $Y$-comodules.  That is, the
diagrams
\[\xymatrix{
{X\otimes X} \ar[r]^-{\phi\otimes X} \ar[d]_{\mu}
& {Y\otimes X\otimes X} \ar[d]^{Y\otimes \mu}
& {Y\otimes Y}
& {X\otimes Y\otimes Y} \ar[l]_-{\psi\otimes Y}
\\
{X} \ar[r]_-{\phi}
& {Y\otimes X}
& {Y} \ar[u]^{\delta}
& {X\otimes Y} \ar[l]^-{\psi} \ar[u]_{X\otimes \delta}
}\]
commute.
\end{prop}
\begin{proof}
The left-hand square follows from the commutativity of 
\[\xymatrix@C=40pt{
{X\otimes X}  \ar[r]^-{X\otimes \phi} \ar[d]_{\mu}
\ar@/^3pc/[rrr]_{\phi\otimes X}
& {X\otimes Y\otimes X} \ar[r]^-{\phi\otimes Y\otimes X}
\ar[d]_{\psi\otimes X}
& {Y\otimes X\otimes Y\otimes X} \ar[d]_{Y\otimes \psi\otimes X}
& {Y\otimes X\otimes X} \ar[l]_-{Y\otimes  X\otimes \phi} \ar[d]_{Y\otimes \mu}
\\
{X} \ar[r]_-{\phi}
& {Y\otimes X} \ar[r]_-{\delta\otimes X} \ar@{=}[dr]
& {Y\otimes Y\otimes X}  \ar[d]|{Y\otimes \epsilon \otimes X} 
& {Y\otimes X}  \ar[l]^-{Y\otimes \phi} \ar@{=}[dl]
\\
&& {Y\otimes X}
}\]
The top diamond commutes by naturality of the monoidal product.  The
three squares along the middle row commute by (3a), (3b), and (3a)
respectively.  The lower triangles follow from the unit identities of
(1b) and (2b).

The commutativity of the right-hand square is proved similarly.
\end{proof}

\subsection*{Definition of right Frobenius pair}

We may similarly define the notion of a \dfn{right Frobenius pair} by
interchanging ``left'' and ``right'' suitably; thus, a right Frobenius pair in
$(\cc,\otimes,\one)$ is precisely a left Frobenius pair in
$(\cc,\otimes^\op,\one)$, where $A\otimes^\op B\defeq B\otimes A$.

\subsection*{Definition of commutative Frobenius pair}

If $(\cc,\otimes,\one)$ is a symmetric monoidal category, we define
a \dfn{commutative Frobenius pair} in $\cc$ to be a left Frobenius pair such that
$\mu$ and $\delta$ are \emph{commutative} product and coproduct,
respectively.  A commutative Frobenius pair will necessarily be both a left
and right Frobenius pair.

\subsection*{Frobenius objects are Frobenius pairs}

Suppose that a left Frobenius pair can be described by
\[
(X,Y,\eta,\mu,\psi,\epsilon,\delta,\phi)=(A,A,\eta,\mu,\mu,\epsilon,\delta,\delta).
\]
Then we see that $(A,\eta,\mu,\epsilon,\delta)$ is precisely what is
usually termed a \dfn{Frobenius object} in $\cc$.  In particular,
axioms (2a) and (2b) are redundant, being in this case consequences of
(1a) and (1b).  Note that this object will also be a right Frobenius
pair.

\subsection*{Frobenius pairs and dualizability}

A Frobenius pair always provides a pair of dualizable objects.  Thus,
if $(X,Y,\eta,\mu,\psi,\epsilon,\delta,\phi)$ is a left Frobenius
pair, we may define maps 
\[
\alpha\defeq \phi\circ \eta \colon \one\ra Y\otimes X,\qquad
\beta\defeq \epsilon\circ\psi \colon X\otimes Y\ra \one,
\]
and we have the following proposition.
\begin{prop}\label{prop:axiom-4}
Given a left Frobenius pair in $\cc$ as above, the data
$(X,Y,\alpha,\beta)$ makes $X$ 
into a \dfn{left dualizable object} of $\cc$.  That is, the composites
\[
X=X\otimes \one \xra{X\otimes \alpha} X\otimes Y\otimes X
\xra{\beta\otimes X} \one\otimes X=X
\]
and
\[
Y=\one\otimes Y \xra{\alpha\otimes Y} Y\otimes X\otimes Y
\xra{Y\otimes \beta} Y\otimes \one=Y
\]
are the identity maps of $X$ and $Y$ respectively.  Thus, we obtain
\dfn{duality 
isomorphisms}
\[
\lambda\colon \Hom_\cc(X\otimes U,V) \xra{\sim} \Hom_\cc(U,Y\otimes V),
\]
and
\[
\rho\colon \Hom_\cc(U\otimes Y,V) \xra{\sim} \Hom_\cc(U,V\otimes X),
\]
natural in objects $U$ and $V$ of $\cc$.
\end{prop}
\begin{proof}
Straightforward from the definitions.  Explicitly,
\[
\lambda(f) =(Y\otimes f)\circ (\phi\eta\otimes U),\qquad
\lambda^{-1}(g) = (\epsilon\psi\otimes V)\circ(X\otimes g),
\]
and
\[
\rho(f) = (f\otimes X)\circ (U\otimes \phi\eta),\qquad
\rho^{-1}(g) = (V\otimes \epsilon \psi)\circ (g\otimes Y).
\]
Note that only properties (3a),
(3b),
and the unit 
and counit 
conditions from (1a), (1b), (2a), (2b) are required; the maps $\delta$
and $\mu$ are not used in this proof.
\end{proof}

\begin{cor}\label{cor:duality-correspondences}
Given a left Frobenius pair in $\cc$ as above,  its duality
isomorphisms give correspondences
\[\xymatrix@R=5pt{
{\Hom_\cc(X\otimes X,X)}
\ar@{<->}[r]^{\lambda}
& {\Hom_\cc(X,Y\otimes X)}
\ar@{<->}[r]^{\rho}
& {\Hom_\cc(X\otimes Y,Y)}
\ar@{<->}[r]^{\lambda}
& {\Hom_\cc(Y,Y\otimes Y)}
\\
{\mu} \ar@{<->}[r]
& {\phi} \ar@{<->}[r]
& {\psi} \ar@{<->}[r]
& {\delta}
}\]
and
\[\xymatrix@R=5pt{
{\Hom_\cc(\one,X)}
\ar@{<->}[r]^{\rho}
& {\Hom_\cc(Y,\one)}
\\
{\eta} \ar@{<->}[r]
& {\epsilon}
}\]
\end{cor}
\begin{proof}
Straightforward, using the axioms and \eqref{prop:axiom-4}.
\end{proof}

\begin{rem}
If $\cc$ is a \emph{symmetric closed monoidal category}, and so
admits an internal function 
object $\uHom({-},{-})$ satisfying $\Hom_{\cc}(Z\otimes A,B)\approx
\Hom_{\cc}(A,\uHom(Z,B))$, then dualizability of $X$
gives rise to 
isomorphisms
\[
\uHom(X\otimes U,V)\approx \uHom(U,Y\otimes V).
\]
In particular, by taking $U=V=\one$ we obtain
an isomorphism $\uHom(X,\one)\approx Y$.
\end{rem}

\section{Atiyah duality}

In this section, we associate to a compact smooth manifold $M$ a
Frobenius pair in the stable homotopy category with objects
$\Sigma^\infty_+M$ and $M^{-\tau}$.  We do this simply by constructing
the relevant objects and maps, and showing that the appropriate
diagrams 
commute up to homotopy.

We will use the following constructions and notation.
\begin{itemize}
\item For a vector bundle $V\ra M$, we write $M^V$ for the associated
  Thom space; we may regard $M^V$ as a spectrum by applying
  $\Sigma^\infty$.  

\item In particular, $M^0\approx \Sigma^\infty(M_+)$, and $(*)^0$ is
  the stable $0$-sphere, which we denote by $\one$.

\item Given an embedding $f\colon M\looparrowright N$ of manifolds, we
  write $\nu(f)\ra M$ for the normal bundle of the embedding.

\item For an embedding $f\colon M\looparrowright N$, we write
  $\widehat{f}\colon N^0\ra M^{\nu(f)}$ for the associated collapse
  map. 

\item More generally, given a sequence of embeddings $f\colon
  M\looparrowright N$ and $g\colon N\looparrowright P$, we obtain a
  collapse map of $N$ to $M$ with respect to the ambient space $P$, of
  the form
\[
\widehat{f}^{\nu(g)} \colon N^{\nu(g)} \ra M^{\nu(gf)}.
\]

\item Given embedding $f\colon M\looparrowright N$
  and a vector bundle $V\ra N$, we obtain a collapse map
  $\widehat{f}^V\colon N^V\ra 
  M^{f^*V\oplus \nu(f)}$ by regarding $N\subset V$ as the $0$-section,
  so that $f^*V\oplus \nu(f)$ is equivalent to the normal bundle of
  $M\subset V$;
  we may think of the map $\widehat{f}^V$ as a collapse map
  ``twisted''by the bundle $V$.

\item We extend our Thom space notation to \emph{Thom spectra of
    virtual bundles}. 
  Thus, $M^{V-W}$ denotes the Thom spectrum of the virtual bundle
  $V-W$.  Likewise, given any smooth map $f\colon M\ra N$ and virtual
  bundle $V\ra N$, we can associate maps $f^V\colon M^{f^*V}\ra N^V$
  and $\widehat{f}^V\colon N^V\ra M^{\nu(f)\oplus V}$.  
\end{itemize}

Especially significant is the following observation.

\subsection*{Homotopy invariance of the collapse map} 
The construction of collapse maps is \emph{homotopy invariant}
  with respect to 
  smooth isotopies.  That is, given smooth isotopies
  $f_0\sim f_1\colon M\looparrowright N$ and $g_0\sim g_1\colon
  N\looparrowright P$ we obtain homotopy equivalences $N^{\nu(g_0)}
  \approx N^{\nu(g_1)}$ and $M^{\nu(g_0f_0)}\approx M^{\nu(g_1f_1)}$,
  which fit in a homotopy commutative diagram
\[\xymatrix{
{M^{\nu(g_0f_0)}} \ar@{<->}[d]_{\sim}
& {N^{\nu(g_0)}} \ar@{<->}[d]^{\sim} \ar[l]_-{\widehat{f_0}^{\nu(g_0)}}
\\
{M^{\nu(g_1f_1)}}
& {N^{\nu(g_1)}} \ar[l]^-{\widehat{f_1}^{\nu(g_1)}}
}\]
We will usually apply this in cases where the \emph{space of choices} of
suitable isotopies turns out to be contractible (or at least, with
connectivity approaching $\infty$ as some parameter $B\to\infty$).
When we can restrict to such a space of choices for our isotopies, it
is clear that all such choices lead to the \emph{same} (up to
homotopy!) homotopy
equivalences $N^{\nu(g_0)}\approx N^{\nu(g_1)}$ and
$M^{\nu(g_0f_0)}\approx M^{\nu(g_1f_1)}$, and thus we may
unambiguously identify these objects (in the homotopy category of
spaces, or spectra as may be); with respect to this identification,
the collapse maps $\widehat{f_0}^{\nu(g_0)}$ and
$\widehat{f_1}^{\nu(g_1)}$ are the same up to homotopy.

\subsection*{The Frobenius pair structure on the stabilization of a smooth manifold}

The Frobenius pair associated to a smooth compact manifold $M$ consists of
spectra $X=M^{-\tau}$ and $Y=M^0$, and maps
\begin{align*}
  \epsilon \colon & M^0\ra \one,
&  \eta \colon & \one \ra M^{-\tau},
\\
  \delta \colon & M^0 \ra M^0\sm M^0,
&  \mu  \colon & M^{-\tau}\sm M^{-\tau}\ra M^{-\tau},
\\
  \phi\colon & M^{-\tau} \ra M^0\sm M^{-\tau},
&   \psi\colon & M^{-\tau}\sm M^0\ra M^0.
\end{align*}
This data will satisfy the axioms for a Frobenius pair in the homotopy
category of spectra.  In brief, the maps in the left-hand column
($\epsilon$, $\delta$, and $\phi$) 
are stabilizations of certain maps between manifolds (or bundles), while
the maps in the right-hand column
($\eta$, $\mu$, and $\psi$) are obtained as Pontryagin-Thom collapse
maps associated to certain embeddings.

We take up the definition of each map, and a sketch of the proofs of
each axiom, in turn.

\subsection*{The maps $\epsilon$ and $\delta$}  These are just the maps
$p^0\colon M^0\ra (*)^0\approx \one$ and $d^0\colon M^0\ra (M\times
M)^0\approx 
M^0\sm M^0$ of
spectra obtained from the projection and diagonal maps $p\colon M\ra
*$ and $d\colon M\ra M\times M$ 

\subsection*{Axiom (1b).}  Clear.

\subsection*{The map $\eta$.}  
Choose an embedding $j\colon M\looparrowright \R^B$ into a Euclidean
space of large 
dimension.  We obtain the Pontryagin-Thom collapse map
\[
\widehat{j}\colon S^B\approx (\R^B)^0 \ra M^{\nu(j)}.
\]
For $B$ sufficiently large, any two embeddings are isotopic; in fact,
as $B\to\infty$, the space of embeddings becomes contractible.  Thus,
this construction produces a well-defined stable homotopy class of
maps
\[
\eta=\widehat{p}\colon \one\ra M^{-\tau};
\]
we write $-\tau$ for the virtual bundle $\nu(j)-\underline{B}$.

\subsection*{The map $\phi$.}  
Choose an embedding $j\colon M\looparrowright \R^B$, and thus a normal
bundle $\nu(j)$ over $M$.  Pulling back the bundle $0\times \nu(j)$
over $M\times M$ along the diagonal map induces a map
\[
d^{0\times \nu(j)}\colon M^{\nu(j)} \ra (M\times M)^{0\times \nu(j)} 
\]
on Thom spaces.  Stabilizing, we obtain $\phi=d^{0\times
  (-\tau)}\colon M^{-\tau}\ra M^0\sm M^{-\tau}$.

\subsection*{Axiom (2b).}  This is standard.

\subsection*{The map $\mu$.}  Pick any pair of embeddings $j_1,j_2\colon
M\looparrowright \R^B$, and consider the sequence of embeddings
\[
M\xra{d} M\times M \xra{j_1\times j_2} \R^B\times \R^B.
\]
We obtain a collapse map associated to the embedding $d$ inside the
ambient space $\R^B\times \R^B$, of the form
\[
\widehat{d}^{\nu(j_1)\times \nu(j_2)}\colon (M\times
M)^{\nu(j_1)\times \nu(j_2)} \ra M^{\nu((j_1,j_2))}. 
\]
As noted, all choices of embedding of $M\looparrowright\R^\infty$ live
in a contractible parameter space; thus after stabilizing, the
embeddings $j_1$, $j_2$, and $(j_1,j_2)$ give the same Thom spectrum
$M^{-\tau}$, and we obtain a map
$\mu=\widehat{d}^{(-\tau)\times (-\tau)}\colon M^{-\tau}\sm M^{-\tau}
\approx (M\times M)^{(-\tau)\times (-\tau)} \ra M^{-\tau}$.  (Compare
\cite{cohen-multiplicative-thom}, where the map $\mu$ is constructed
and refined to a strictly commutative multiplication on the spectrum
(without strict unit).) 

\subsection*{Axiom (1a).}  To prove the unit identity $\mu\circ
(\eta\otimes X)=1_X$, consider the sequence of embeddings
\[
M\xra{d} M\times M\xra{j_1\times \id} \R^B\times M \xra{\id\times j_2}
\R^B\times \R^B.
\]
The collapse map obtained from the embedding $(j_1,\id)=(j_1\times
\id)\circ d \colon M\looparrowright \R^B\times M$  inside the
ambient space $\R^B\times \R^B$ is a composite
\[
S^B\sm M^{\nu(j_2)} \ra (M\times M)^{\nu(j_1)\times \nu(j_2)} \ra
M^{\nu(j_1,j_2)},
\]
which can be written as
\[
S^B \sm M^{\nu(j_2)} \xra{\widehat{j_1}\sm \id} M^{\nu(j_1)}\sm
M^{\nu(j_2)} \xra{\widehat{d}^{\nu(j_1)\times \nu(j_2)}} M^{\nu(j_1,j_2)},
\]
which when stabilized realizes $\one \sm M^{-\tau} \xra{\eta \sm\id}
M^{-\tau}\sm M^{-\tau}\xra{\mu} M^{-\tau}$.  

On the other hand, the
map $j_1\colon M\ra \R^B$ is homotopic through smooth maps to a
constant map $0\colon M\ra \R^B$, and such a homotopy produces an
isotopy $(j_1,\id)\sim (0,\id)$ of embeddings $M\ra \R^B\times M$.
(The space of smooth maps $j_1\colon M\ra \R^B$ is contractible, so there
is no ambiguity created by the 
choice of homotopy.) 
The sequence of embeddings
\[
M\xra{(0,\id)} \R^B\times M \xra{\id\times j_2} \R^B\times \R^B
\]
induces a collapse map
\[
S^B\sm M^{\nu(j_2)} \ra M^{\nu(0,j_2)}\approx M^{\underline{B}\oplus
  \nu(j_2)}
\]
which is manifestly homotopic to the identity.  According to the homotopy
invariance of the collapse map with respect to isotopies, this map
may be naturally identified up to homotopy with the composite
$\widehat{d}^{\nu(j_1)\times \nu(j_2)}\circ (\widehat{j_1}\sm \id)$.  

The unit identity $\mu\circ (X\otimes \eta)=1_X$ is proved similarly.

To prove the associativity identity, choose embeddings
$j_1,j_2,j_3\colon M\looparrowright \R^B$, and consider the
commutative square of 
embeddings
\[\xymatrix{
{M} \ar[r]^-{d} \ar[d]_{d}
& {M\times M} \ar[d]^{d\times \id}
\\
{M\times M} \ar[r]_-{\id\times d}
& {M\times M\times M}
}\]
The diagram of associated collapse maps of the submanifolds inside
$\R^B\times \R^B\times \R^B$ has the form 
\[\xymatrix{
{M^{\nu((j_1,j_2,j_3))}} 
&&& {(M\times M)^{\nu((j_1,j_2))\times \nu(j_3)}}
\ar[lll]_-{\widehat{d}^{\nu((j_1,j_2))\times \nu(j_3)}}
\\
{(M\times M)^{\nu(j_1)\times \nu((j_2,j_3))}}
\ar[u]^{\widehat{d}^{\nu(j_1)\times \nu((j_2,j_3))}}
&&& {(M\times M\times M)^{\nu(j_1)\times \nu(j_2)\times \nu(j_3)}}
\ar[lll]^-{\widehat{\id\times d}^{\nu(j_1)\times \nu(j_2)\times
    \nu(j_3)}} \ar[u]_{\widehat{d\times \id}^{\nu(j_1)\times
    \nu(j_2)\times \nu(j_3)}}
}\]
which after stabilizing is the desired commutative diagram.

\subsection*{The map $\psi$.}
Pick an embedding $j\colon M\looparrowright \R^B$, and consider the
sequence of embeddings
\[
M\xra{d} M\times M\xra{j\times \id} \R^B\times M.
\]
Forming the collapse map of $d$ with respect to the ambient space $\R^B\times
M$ gives 
\[
\widehat{d}^{\nu(j)\times 0} \colon (M\times M)^{\nu(j)\times 0} \ra
M^{\nu((j,\id))}.
\]
By choosing any smooth homotopy of $j\colon M\ra \R^B$ to a constant
map, we obtain an isotopy $(j,\id)\sim (0,\id)$, which provides a
bundle equivalence $\nu((j,\id))\approx \nu((0,\id))=\underline{B}$.
That is, we obtain a map $M^{\nu(j)}\sm M^0\ra M^{\underline{B}}$; we
let $\psi=\widehat{d}^{(-\tau)\times 0} \colon M^{-\tau}\sm M^0\ra
M^0$ be the map obtained after stabilization.

\subsection*{Axiom (2a).}
Choose an embedding $j\colon M\looparrowright \R^B$.  
Since the composite 
\[
M\xra{d} M\times M\xra{j\times \id} \R^B\times M
\]
is isotopic to $(0,\id)\colon M\looparrowright \R^B\times M$ by means
of a smooth homotopy $j\sim 0$, the composite of collapse maps 
\[S^B\sm M \xra{\widehat{j}\sm \id} M^{\nu(j)}\sm M^0
\xra{\widehat{d}^{\nu(j)\times 0}} M^{\nu(j,\id)}
\]
is homotopic to the collapse map of $(0,\id)\colon M\ra \R^B\times M$,
which is homotopic to the identity map of $S^B\sm M^0$; this proves
the unit identity.  

Likewise, the commutative square of diagonal embeddings into $M\times
M\times M\subset \R^B\times \R^B\times M$ induces a commutative square
of collapse maps (relative to the ambient space $\R^B\times \R^B\times
M$), 
\[\xymatrix{
{M^{\nu((j_1,j_2,\id))}} 
&&& {(M\times M)^{\nu(j_1,j_2)\times 0}}
\ar[lll]_-{\widehat{d}^{\nu(j_1,j_2)\times 0}}
\\
{(M\times M)^{\nu(j_1)\times \nu((j_2,\id))}}
\ar[u]^{\widehat{d}^{\nu(j_1)\times \nu((j_2,\id))}}
&&& {(M\times M\times M)^{\nu(j_1)\times \nu((j_2)\times 0}}
\ar[lll]^-{\widehat{\id \times d}^{\nu(j_1)\times \nu(j_2)\times 0}}
\ar[u]_{\widehat{d\times \id}^{\nu(j_1)\times \nu(j_2)\times 0}}
}\]
which is precisely the associativity identity.

\subsection*{A transversality diagram.}
Suppose given a commutative diagram of manifolds
\[\xymatrix{
{A_1} \ar[r]^{h_1} \ar[d]_{f}
& {B_1} \ar[d]^{g} \ar[r]^{k_1}
& {C_1}
\\
{A_2} \ar[r]_{h_2}
& {B_2} \ar[r]_{k_2}
& {C_2}
}\]
in which $h_1,h_2,k_1,k_2$ are embeddings, and $A_1$ is the
transversal intersection of $A_2$ along $g$.  Furthermore, suppose we
are given a bundle equivalence $\beta\colon \nu(k_1)\xra{\sim}
g^*\nu(k_2)$ over $B_1$, which induces a bundle equivalence
$\alpha\colon \nu(k_1h_1)\xra{\sim} f^*\nu(k_2h_2)$ over $A_1$ (using
the evident equivalence $\nu(h_1)\approx f^*\nu(h_2)$.  

Then we obtain a homotopy commutative diagram 
\[\xymatrix{
{A_1^{\nu(k_1h_1)}} \ar[d]_{f^{\nu(k_2h_2)}}
& {B_1^{\nu(k_1)}} \ar[d]^{g^{\nu(k_2)}} \ar[l]_-{\widehat{h_1}^{\nu(k_1)}}
\\
{A_2^{\nu(k_2h_2)}} 
& {B_2^{\nu(k_2)}} \ar[l]^-{\widehat{h_2}^{\nu(k_2)}}
}\]
in which the vertical maps are inclusions of Thom spaces induced by
the pullback squares
\[\xymatrix{
{\nu(k_1h_1)} \ar[r]^{\alpha} \ar[d]
& {\nu(k_2h_2)} \ar[d]
& {\nu(k_1)} \ar[r]^{\beta} \ar[d]
& {\nu(k_2)} \ar[d]
\\
{A_1} \ar[r]_{f} 
& {A_2} 
& {B_1} \ar[r]_{g}
& {B_2}
}\]
and the
horizontal maps are collapse maps.

\subsection*{Axiom (3a).}
Consider 
\[\xymatrix{
{M} \ar[r]^-{d} \ar[d]_{d}
& {M\times M} \ar[d]^{\id\times d} \ar[rr]^{j_1\times j_3}
&& {\R^B\times \R^B}
\\
{M\times M} \ar[r]_-{d\times \id}
& {M\times M\times M} \ar[rr]_{j_1\times \id\times j_3}
&& {\R^B\times M\times \R^B}
}\]
where $j_1,j_3\colon M\looparrowright \R^B$ are embeddings, where
\[
\beta\colon \nu(j_1)\times \nu(j_3) \xra{\sim} (\id \times
d)^*(\nu(j_1)\times 0\times \nu(j_3)) 
\]
is the obvious bundle isomorphism, and
\[
\alpha\colon \nu((j_1,j_3)) \xra{\sim} d^*(\nu((j_1,\id))\times
\nu(j_3))
\]
is the bundle isomorphism obtained by pulling back $\beta$ along the
horizontal maps.
We obtain a commutative square
\[\xymatrix{
{M^{\nu((j_1,j_3))}} \ar[d]_{d^{\nu((j_1,\id))\times \nu(j_3)}} 
&&& {(M\times M)^{\nu(j_1)\times \nu(j_3)}}
\ar[d]^{(\id\times d)^{\nu(j_1)\times 0\times \nu(j_3)}}
\ar[lll]_-{\widehat{d}^{\nu(j_1)\times \nu(j_3)}}
\\
{(M\times M)^{\nu((j_1,\id))\times \nu(j_3)}}
&&& {(M\times M\times M)^{\nu(j_1)\times 0\times \nu(j_3)}}
\ar[lll]^-{\widehat{d\times \id}^{\nu(j_1)\times 0\times \nu(j_3)}}
}\]
Choose a smooth homotopy of $j_1\colon M\ra \R^B$ to a constant map,
thus producing isotopies $(j_1,j_3)\sim (0,j_3)$ of embeddings
$M\looparrowright \R^B\times \R^B$ and $(j_1,\id)\times j_3 \sim
(0,\id)\times j_3$ of embeddings $M\times M\looparrowright \R^B\times
M\times \R^B$.  In addition, we may use this same homotopy to form a
$1$-parameter family 
of bundle maps over $d\colon M\ra M\times M$, between $\alpha$ and the
evident bundle isomorphism
\[
\alpha'\colon \nu((0,j_3))\xra{\sim}
d^*(\nu((0,\id))\times \nu(j_3)).
\]
Thus, $d^{\nu((j_1,\id))\times \nu(j_3)} \colon M^{\nu((j_1,j_3))} \ra
(M\times M)^{\nu((j_1,\id))\times \nu(j_3)}$ is homotopic to 
\[
M^{\underline{B}\oplus \nu(j_3)} \ra M^{\underline{B}}\sm M^{\nu(j_3)}.
\]
After stabilizing, the above diagram is the homotopy commutative
diagram
\[\xymatrix{
{M^{-\tau}}  \ar[d]_{\phi}
& {M^{-\tau}\sm M^{-\tau}} \ar[d]^{\id\sm \phi} \ar[l]_-{\mu}
\\
{M^0\sm M^{-\tau}} 
& {M^{-\tau}\sm M^0\sm M^{-\tau}} \ar[l]^-{\psi\sm \id}
}\]

\subsection*{Axiom (3b).}
Consider 
\[\xymatrix{
{M} \ar[r]^-{d} \ar[d]_{d}
& {M\times M} \ar[d]^{d\times \id} \ar[rr]^{j\times \id}
&& {\R^B\times M}
\\
{M\times M} \ar[r]_-{\id\times d}
& {M\times M\times M} \ar[rr]_{\id\times j\times \id}
&& {M\times \R^B\times M}
}\]
where $j\colon M\looparrowright \R^B$ is an embedding, where
\[
\beta\colon \nu(j)\times 0 \xra{\sim} (d \times
\id)^*(0\times \nu(j)\times 0) 
\]
is the obvious bundle isomorphism, and
\[
\alpha\colon \nu((j,\id)) \xra{\sim} d^*(0\times \nu(j,\id))
\]
is the bundle isomorphism obtained by pulling back $\beta$ along the
horizontal maps.
We obtain a commutative square
\[\xymatrix{
{M^{\nu((j,\id))}} \ar[d]_{d^{0\times \nu((j,\id))}} 
&&& {(M\times M)^{\nu(j)\times 0}}
\ar[d]^{(d\times \id)^{0\times \nu(j)\times 0}}
\ar[lll]_-{\widehat{d}^{\nu(j)\times 0)}}
\\
{(M\times M)^{0\times \nu((j,\id))}}
&&& {(M\times M\times M)^{0\times \nu(j)\times 0}}
\ar[lll]^-{\widehat{\id \times d}^{0\times \nu(j)\times 0}}
}\]
Choose a smooth homotopy of $j\colon M\ra \R^B$ to a constant map,
thus producing isotopies $(j,\id)\sim (0,\id)$ of embeddings
$M\looparrowright \R^B\times M$ and $\id\times (j,\id) \sim
\id \times (0,\id)$ of embeddings $M\times M\looparrowright M\times \R^B\times
M$.  In addition, we may use this same homotopy to form a
$1$-parameter family 
of bundle maps over $d\colon M\ra M\times M$, between $\alpha$ and the
evident bundle isomorphism
\[
\alpha'\colon \underline{B}\xra{\sim}
d^*(0\times\underline{B}).
\]
Thus, $d^{0\times \nu((j,\id))} \colon M^{\nu((j,\id))} \ra
(M\times M)^{0\times \nu((j_1,\id))}$ is homotopic to 
\[
M^{\underline{B}} \ra M^{0}\sm M^{\underline{B}}.
\]
After stabilizing, the above diagram is the homotopy commutative
diagram
\[\xymatrix{
{M^{0}}  \ar[d]_{\delta}
& {M^{-\tau}\sm M^{0}} \ar[d]^{\phi\sm \id} \ar[l]_-{\psi}
\\
{M^0\sm M^{0}} 
& {M^{0}\sm M^{-\tau}\sm M^{0}} \ar[l]^-{\id \sm\psi} 
}\]

\subsection*{Commutativity}

Let $\sigma\colon X\sm Y\ra Y\sm X$ denote the symmetry of the smash
product in the homotopy category of spectra.
It is immediate that $\sigma\circ \delta\approx \delta$, from the
symmetry of the diagonal embedding $d\colon M\ra M\times M$.
To show that $\mu\circ\sigma \approx \mu$, it suffices to note that if
$j_1,j_2\colon M\looparrowright \R^B$ are embeddings, then for $B$
sufficiently large the embeddings $(j_1,j_2), (j_1,j_1)\colon
M\looparrowright \R^B\times\R^B$ are isotopic, 

\subsection*{Atiyah duality}

We have shown the following.
\begin{prop}
Let $M$ be a smooth compact manifold.  Then the pair of spectra
$\Sigma^\infty M_+$ and $M^{-\tau}$ admit the structure of a
commutative Frobenius pair
in the homotopy category of spectra.
\end{prop}

As a consequence of general properties of Frobenius pairs, we recover
Atiyah duality.

\begin{cor}[Atiyah]
There is a weak equivalence between $\Sigma^\infty M_+$ and the
Spanier-Whitehead dual of $M^{-\tau}$.
\end{cor}

\begin{rem}
We can extend the above arguments to deal with duality for manifolds
with boundary.  Thus, if $M$ is a smooth compact manifold with boundary
$\partial M=N$, then can we take
\[
X=(M/N)^{-\tau},\qquad Y=M^0,
\]
and define maps
\begin{align*}
  \epsilon \colon & M^0\ra \one,
&  \eta \colon & \one \ra (M/N)^{-\tau},
\\
  \delta \colon & M^0 \ra M^0\sm M^0,
&  \mu  \colon & (M/N)^{-\tau}\sm (M/N)^{-\tau}\ra (M/N)^{-\tau},
\\
  \phi\colon & (M/N)^{-\tau} \ra M^0\sm (M/N)^{-\tau},
&   \psi\colon & (M/N)^{-\tau}\sm M^0\ra M^0,
\end{align*}
which define a Frobenius pair.
\end{rem}


\end{document}